\documentclass[a4paper,oneside,11pt]{amsart}
\usepackage{amssymb,amsmath,enumerate}
\usepackage{hyperref}
\hypersetup{colorlinks,bookmarksopen,bookmarksnumbered,
  citecolor=blue,linkcolor=black,pdfstartview=FitH}

\newtheorem{theorem}{Theorem}[section]
\newtheorem{lemma}[theorem]{Lemma}
\newtheorem{proposition}[theorem]{Proposition}
\newtheorem{corollary}[theorem]{Corollary}
\theoremstyle{definition}
\newtheorem{definition}[theorem]{Definition}

\newtheorem{remark}[theorem]{Remark}

\newtheorem{question}[theorem]{Question}

\DeclareMathOperator*{\sotlim}{{\mbox{\scshape sot}}-lim}
\DeclareMathOperator{\ran}{Ran}
\newcommand\supp{\mathop{\rm supp}}
\let\spn\sspp
\newcommand\nph{\varphi}
\newcommand\al{\alpha}

\newcommand\ocl{\mathop{{\rm cl}_{\omega}}}
\newcommand\ointer{\mathop{{\rm int}_{\omega}}}
\newcommand\sqinter{\mathop{{\rm sqint}_{\omega}}}

\newcommand\cb{\mathop{\rm cb}}
\newcommand{\cl}[1]{\mathcal{#1}}
\newcommand{\bb}[1]{\mathbb{#1}}
\newcommand{\bN}{\bb N}

\newcommand{\B}{\cl B}
\renewcommand{\S}{\cl S}

\newcommand{\D}{\cl D}
\newcommand\esssup{\mathop{\rm esssup}}
\newcommand\Ref{\mathop{\rm Ref}}

\begin{document}

\title{Positive extensions of Schur multipliers}

\author[R.H. Levene]{Rupert H. Levene}
\address{School of Mathematics and Statistics,
  University College Dublin, Belfield, Dublin~4, Ireland}
\email{rupert.levene@ucd.ie}

\author[Y.-F. Lin]{Ying-Fen Lin}
\address{Pure Mathematics Research Centre,
  Queen's University Belfast, Belfast BT7 1NN, United Kingdom}
\email{y.lin@qub.ac.uk}

\author[I. Todorov]{Ivan G. Todorov}
\address{Pure Mathematics Research Centre,
  Queen's University Belfast, Belfast BT7 1NN, United Kingdom}
\email{i.todorov@qub.ac.uk}

\date{%
23 December 2016}

\begin{abstract}
  We introduce partially defined Schur multipliers and obtain
  necessary and sufficient conditions for the existence of extensions
  to fully defined positive Schur multipliers, in
  terms of operator systems canonically associated with their domains.
  We use these results to study the problem of extending a
  positive definite function defined on a symmetric subset of a locally compact
  group to a positive definite function defined on the whole group.
\end{abstract}

\maketitle

\section{Introduction}\label{s_intro}

The problem of completing a partially defined
matrix to a fully defined positive matrix 
has attracted considerable attention in the literature (see {\it e.g.} \cite{bn_sz, dg,gjsw,pps}).  
Given an $n$ by $n$ matrix,
only a subset of whose entries are specified, this problem asks whether the remaining
entries can be determined so as to yield a positive matrix. 
The set $\kappa$ of pairs $(i,j)$ for which the $(i,j)$-entry is specified is called the
pattern of the initial matrix; to avoid trivialities,
$\kappa$ is assumed to be symmetric and to contain the diagonal.  One
may then consider the operator system~$\cl S(\kappa)$ of all (fully
specified) matrices supported by $\kappa$.
The operator systems arising in this way are precisely the
operator subsystems of the space $M_n$ of all $n$ by $n$ complex matrices, which
are also bimodules over the algebra of all diagonal matrices. 
In~\cite{pps}, using operator space methods, Paulsen, Power and Smith
 formulated necessary and sufficient conditions for 
the existence of positive
completions for a given pattern $\kappa$, in terms of
$\cl S(\kappa)$, and related such completions to positivity and 
extendability of associated Schur multipliers with domain $\kappa$. 


In this paper, we study the corresponding extension problem in
infinite dimensions. More precisely, we replace the Hilbert space
$\bb{C}^n$ with the Hilbert space $L^2(X,\mu)$ for some
measure space~$(X,\mu)$, and the algebra of diagonal matrices with the
maximal abelian selfadjoint algebra~$\cl D \equiv L^\infty(X,\mu)$.
Given a suitable measurable subset~$\kappa\subseteq X\times X$, we define
a weak* closed $\cl D$-bimodule $\cl S(\kappa)$, canonically associated with
$\kappa$, and introduce the notion of a Schur multiplier with
domain~$\kappa$. We study the problem of extending such a (partially defined) Schur
multiplier to a positive Schur multiplier defined on all of~$X\times X$, 
and relate it to the positivity structure of $\cl S(\kappa)$.

Our motivation is two-fold. Firstly, we will see in Section \ref{s_epdf} that 
the problem we consider is closely related to 
the problem of extending partially defined positive definite functions
on locally compact groups. The latter problem has been studied  
in a variety of contexts and there is a rich bibliography 
on its modern aspects as well as its connections with classical problems and applications 
(see \cite{bn_sz}, \cite{bn}, \cite{bt}, \cite{devinatz}, \cite{g}, \cite{mcm} and the references therein). Since the main interest 
here lies in infinite groups, the passage to infinite dimensions becomes necessary. 

Secondly, bimodules over continuous maximal abelian selfadjoint algebras 
(masas, for short)
have been instrumental in a number of contexts in Operator Algebras. 
Introduced by Arveson in \cite{a}, they have proved useful in topics as diverse as 
spectral synthesis and uniqueness sets in Harmonic Analysis \cite{st}, 
closability of multipliers on Fourier algebras \cite{stt}, finite rank approximations \cite{eks}
and structure of idempotents \cite{kp}, among others. 
They are also closely related to Schur multipliers 
(see~\cite{Pa} and~\cite{Pi}, as well as \cite{naomi_tt}, where
questions related to positivity of Schur multipliers were studied).
However, masa-bimodules that are simultaneously operator systems
have not received attention to date, despite their
importance in modern Analysis \cite{Pa}. 

The paper is organised as follows. 
In Section~\ref{s_Bell2}, we consider the discrete case and 
formulate a straightforward generalisation of several results in~\cite{pps}, which use a graph
theoretic property of the pattern $\kappa$ called chordality. 
We note that extension results 
for partially defined functions that are not necessarily Schur multipliers,
in terms of chordality, were obtained in \cite{t}. 

In Section~\ref{s_pdom}, we study
measurable versions of the patterns $\kappa$ and their operator
systems.  Although these subsets $\kappa$ can be thought of as
measurable graphs, the passage from a discrete to a
general measure space leads to substantial differences (see {\it e.g.} Corollary~\ref{c_exi}).  
In Section~\ref{s_e}, we formulate necessary and sufficient conditions for the existence of 
an extension
of a partially defined partially positive Schur multiplier 
to a fully defined positive Schur multiplier, in terms of 
approximation of positive operators
by sums of rank one positive operators in the operator system $\cl S(\kappa)$.

In Section~\ref{s_epdf}, we study the problem of 
extending a positive definite function defined on a
symmetric subset $E$ of a locally compact group to a positive definite function
defined on the whole group. 
The special case where $E$ is a closed subgroup has 
attracted considerable attention previously (see {\it e.g.}~\cite{kl}).  
Closely related problems about extension of Herz-Schur multipliers were recently 
considered in~\cite{br_forr}.  
We use the results from Section~\ref{s_e}
to give, in the case of discrete amenable 
groups, a different approach to the result of Bakonyi and Timotin \cite{bt}
concerning positive definite extensions of partially defined functions. 
Our main result here (Theorem \ref{th_lcex}) concerns
(classes of non-discrete) locally compact groups, where we 
formulate a sufficient condition for the existence of 
positive definite extensions in terms of operator approximations.

\smallskip

For a Hilbert space $H$, we denote by $\cl B(H)$ (resp. $\cl K(H)$) 
the space of all bounded linear (resp. compact) 
operators on $H$.
We will often use basic concepts from Operator Space Theory, such as 
complete positivity; we refer the reader to \cite{Pa} for the necessary background. 
As customary, the closure of a subset $\cl T$ in a topology $\tau$ will be denoted by 
$\overline{\cl T}^{\tau}$, and if ${\cl T} \subseteq \cl B(H)$, then we will write ${\cl T}^+$ for 
the set of all positive elements of ${\cl T}$.

\section{The discrete case}\label{s_Bell2}

Let $X$ be a set and let $H = \ell^2(X)$.  With every element $T$
of~$\cl B(H)$, we associate the matrix~$(t_{x,y})_{x,y\in X}$ given by
$t_{x,y} = (Te_y,e_x)$, where $(e_x)_{x\in X}$ is the canonical
orthonormal basis of~$H$. 
For $x,y\in X$, denote by
$E_{x,y}$ the corresponding matrix unit in~$\cl B(H)$ (so that
$E_{x,y} e_y = e_x$).

For $\kappa\subseteq X\times X$, define
\begin{equation}\label{eq_dis}
\cl S(\kappa) = \overline{\spn\{E_{x,y} : (x,y)\in \kappa\}}^{w^*}.
\end{equation}
It is easy to see that an operator~$T\in \cl B(H)$ is in~$\cl
S(\kappa)$ if and only if the corresponding matrix~$(t_{x,y})$ has
$t_{x,y}=0$ whenever $(x,y)$ is in the complement of~$\kappa$.

Throughout this section, we fix an additional (non-trivial) Hilbert space $K$ and let 
$H\otimes K$ be the Hilbert space tensor product of $H$ and $K$. 
The elements $T\in\cl B(H\otimes K)$ can in a similar fashion be regarded as matrices
$(T_{x,y})_{x,y\in X}$, where $T_{x,y}\in \cl B(K)$ is determined by the identity
\[(T_{x,y}\xi,\eta) = (T(e_y\otimes \xi),e_x\otimes \eta),
\quad x,y\in X,\ \xi,\eta\in K.\]

\begin{definition}
  A function $\psi : \kappa\to \cl B(K)$ will be called an \emph{(operator valued)
    Schur multiplier} if the matrix $(t_{x,y}\psi(x,y))_{x,y\in X}$ defines an element 
    of $\cl B(H\otimes K)$ for every $(t_{x,y})_{x,y\in X}\in \cl S(\kappa)$
    (here, we have set $t_{x,y}\psi(x,y) = 0$ provided $(x,y)\not\in \kappa$).
\end{definition}

Let 
\[\cl S_0(\kappa) = \overline{\spn\{E_{x,y} : (x,y)\in \kappa\}}^{\|\cdot\|}.\] 
We have  that $\cl S_0(\kappa)\subseteq \cl K(H)$; let 
$\iota : \cl S_0(\kappa)\to \cl K(H)$ be the inclusion map.

\begin{lemma}\label{r_exte}
(i) The second dual $\iota^{**}$ of $\iota$ is a
completely isometric weak* homeomorphism of $\cl S_0(\kappa)^{**}$ onto $\cl S(\kappa)$.

(ii) Every bounded linear map $\Psi : \cl S_0(\kappa)\to \cl B(H\otimes K)$ has a bounded weak* continuous extension 
to a map $\tilde{\Psi} : \cl S(\kappa)\to \cl B(H\otimes K)$. 

(iii) Every bounded linear map $\Psi : \cl K(H)\to \cl B(H)$ has a bounded weak* continuous extension 
to a map $\tilde{\Psi} : \cl B(H)\to \cl B(H)$. 
Moreover, if $\Psi$ is completely positive (resp.
completely bounded) then $\tilde{\Psi}$ is completely positive (resp. completely bounded).
\end{lemma}
\begin{proof}
(i) The map $\iota^{**}$ is a surjective completely isometric weak* homeomorphism
onto its range in $\cl B(H)$; since $\overline{\cl S_0(\kappa)}^{w^*} = \cl S(\kappa)$, 
we conclude that the range of $\iota^{**}$ is $\cl S(\kappa)$. 

(ii) 
Let $\cl E : \cl B(H\otimes K)^{**}\to \cl B(H\otimes K)$ be the canonical (weak* continuous)
projection; thus, $\cl E(T) = T$ whenever $T\in \cl B(H\otimes K)$. 
Set $\tilde{\Psi} = \cl E\circ \Psi^{**} \circ (\iota^{**})^{-1}$; thus, $\tilde{\Psi} : \cl S(\kappa) \to \cl B(H\otimes K)$
is weak* continuous and its restriction to $\cl S_0(\kappa)$ coincides with $\Psi$.

(iii) The statement follows from (ii) after letting $K = \bb{C}$ and $\kappa = X\times X$. 
The remaining statements follow from the facts that $\cl E$ is unital and completely positive (and hence 
completely bounded) and that $\Psi^{**}$ is completely positive (resp. bounded) provided $\Psi$ is so. 
\end{proof}

\begin{proposition}\label{p_opval}
  Let $\kappa\subseteq X\times X$. If $\psi : \kappa\to \cl B(K)$ is a
  Schur multiplier, then the map
  \[S_{\psi} : \cl S(\kappa)\to \cl B(H\otimes K), \quad
  (t_{x,y})\mapsto \big(t_{x,y} \psi(x,y)\big)\] is bounded and weak* continuous.
\end{proposition}
\begin{proof}
  To show that $S_{\psi}$ is bounded, we use the Closed Graph Theorem. Suppose that $T_n =
  (t^n_{x,y})_{x,y\in X}$ are operators in $\cl S(\kappa)$ such that
  $T_n\to 0$ and $S_{\psi}(T_n)\to S$ in norm as $n\to \infty$, for some $S\in
  \cl B(H\otimes K)$.  Letting $S = (S_{x,y})_{x,y\in X}$, we have
  $t^n_{x,y}\to 0$ and $t^n_{x,y} \psi(x,y)\to S_{x,y}$ in norm for
  each~$(x,y)\in \kappa$. It follows that $S_{x,y} = 0$ for each
  $x,y\in X$, and hence $S = 0$. Thus, $S_{\psi}$ is bounded.

Let $S_{\psi}^0$ be the restriction of $S_{\psi}$ to the subspace $\cl S_0(\kappa)$
defined before Lemma \ref{r_exte}. Let $\Psi$ be the weak* continuous extension of $S_{\psi}^0$
guaranteed by Lemma \ref{r_exte} (ii). 
Fix $T = (t_{x,y})_{x,y\in X}\in \cl S(\kappa)$. For a finite set $F\subseteq X$, let 
$P_F$ be the projection on $H$ whose range has basis $\{e_x : x\in F\}$ and set $T_F = P_F T P_F$.
The net $(T_F)_F$ lies in $\cl S_0(\kappa)$ and $T_F\to_F T$ in the weak* topology. 
For all $x,y\in X$ and all $\xi,\eta\in K$, we have
\begin{eqnarray*}
(\Psi(T)(e_y\otimes\xi),e_x\otimes \eta) 
& = & 
\mbox{w}^*\mbox{-}\lim\mbox{}_F (\Psi(T_F)(e_y\otimes\xi),e_x\otimes \eta)\\
& = & 
\mbox{w}^*\mbox{-}\lim\mbox{}_F (S_{\psi}^0(T_F)(e_y\otimes\xi),e_x\otimes \eta)\\
& = & (t_{x,y}\psi(x,y)\xi,\eta)
= (S_{\psi}(T)(e_y\otimes\xi),e_x\otimes \eta).
\end{eqnarray*}
We conclude that $\Psi = S_{\psi}$, and hence $S_{\psi}$ is weak* continuous.
\end{proof}

Note that if $\psi : \kappa\to \cl B(K)$ is a Schur multiplier then the
range of the map $S_{\psi}$, defined in Proposition \ref{p_opval}, is 
contained in the weak* closed spatial tensor product 
$\cl S(\kappa)\bar\otimes\cl B(K)$,  of $\cl S(\kappa)$ and $\cl B(K)$.

A Schur multiplier $\nph : X\times X\to \cl B(K)$ will be called
\emph{positive} if the map $S_{\nph}$ is positive, that is, if for
every positive operator~$T\in \cl B(H)$, the operator $S_{\nph}(T)\in
\cl{B}(H\otimes K)$ is also positive.  
An application of~\cite[Proposition 1.2]{pps} shows that if $\nph$ is
a positive Schur multiplier then the map $S_{\nph}$ is in fact
completely positive.

Let $\kappa \subseteq X\times X$. It is straightforward to verify that the
subspace $\cl S(\kappa)$ is an operator system ({\it i.e.} a selfadjoint unital
subspace of $\cl B(H)$) if and only if
\begin{enumerate}[(i)]
\item $\kappa$ is symmetric (that is, $(x,y)\in \kappa$ implies that
  $(y,x)\in \kappa$), and

\item $\kappa$ contains the diagonal of $X\times X$ (that is,
  $(x,x)\in \kappa$ for every $x\in X$).
\end{enumerate}
Note that such subsets $\kappa$ can be identified with
(undirected) graphs with vertex set $X$ and no loops:
for distinct elements $x,y\in X$, the subset $\{x,y\}$ is an edge if, by definition,
$(x,y)\in \kappa$. Thus, subsets satisfying properties (i) and (ii) above will hereafter be referred to as graphs.

\begin{definition}
  Let $\kappa\subseteq X\times X$ be a graph and let $\psi\colon
  \kappa\to \cl B(K)$ be a Schur multiplier.
We say that $\psi$ is \emph{partially positive} if for every
  subset $\al\subseteq X$ with $\al \times \al \subseteq \kappa$, the
  Schur multiplier $\psi|_{\al \times \al}$ is positive.
\end{definition}

Let $\kappa\subseteq X\times X$ be a graph.
A Schur multiplier $\nph : X\times X\to \cl B(K)$
will be called an \emph{extension} of the Schur multiplier 
$\psi : \kappa\to \cl B(K)$ if the restriction $\nph|_{\kappa}$ of $\nph$
to $\kappa$ coincides with $\psi$.
We will be interested in the question of when a Schur multiplier 
$\psi : \kappa\to \cl B(K)$ possesses a positive extension. 
Clearly, a necessary condition for this to happen is that 
$\psi$ be partially positive. In Theorem \ref{thm:op-valued}, we will identify 
conditions which ensure that the converse implication holds true.

We say that the vertices $x_1,\dots,x_n$ form a cycle of $\kappa$ (of length $n$) 
if $(x_i,x_{i+1})\in \kappa$ for all $i$ (where addition is performed mod $n$). 
A chord in such a cycle is an edge of the form $(x_i,x_k)$, where $2\leq |i-k|\leq n-2$.
We say that $\kappa$ is \emph{chordal} (see {\it e.g.} \cite{pps}) 
if every cycle of length at least four has a chord.

\begin{theorem}\label{thm:op-valued}
  Let $\kappa\subseteq X\times X$ be a graph. The following conditions are equivalent:
  \begin{enumerate}[(i)]
  \item every partially positive Schur multiplier $\psi: \kappa \to
    \cl B(K)$ has a positive extension;
  \item the graph $\kappa$ is chordal;
  \item  every positive operator in $\cl S(\kappa)$ is the weak* limit of
    sums of rank one positive operators in $\cl S(\kappa)$.
  \end{enumerate}
\end{theorem}
\begin{proof}
We denote by $F$ an arbitrary finite subset of $X$, and let $\kappa_F=\kappa\cap (F\times F)$.

  (i)$\Rightarrow$(ii)   
We claim that condition~(i) is
  satisfied for the graph $\kappa_F$.  Indeed,
  given a partially positive Schur multiplier $\psi_F : \kappa_F \to
  \cl B(K)$, let $\psi : \kappa\to \cl B(K)$ be the extension of $\psi_F$ with 
$\psi(x,y) = 0$ if $(x,y)\in \kappa\setminus
  \kappa_F$. Since~$\psi$ has finite support, $\psi$ is a Schur
  multiplier on~$\kappa$. If $\alpha\subseteq X$ is a set with
  $\alpha\times\alpha\subseteq \kappa$, then
  $S_{\psi|_{\alpha\times\alpha}} = S_{\psi_F|_{(\alpha\cap
      F)\times(\alpha\cap F)}}\oplus 0$. Since~$\psi_F$ is partially
  positive, $S_{\psi_F|_{(\alpha\cap F)\times(\alpha\cap F)}}$ is
  positive, so $S_{\psi|_{\alpha\times\alpha}}$ is positive.  
  Thus, $\psi$ is partially positive. By assumption, $\psi$ has
  a positive extension, whose restriction to $F\times F$ is a
  positive extension of $\psi_F$.  It now follows from~\cite{pps}
  that $\kappa_F$ is chordal, and since 
this holds for every finite set $F$,   
we conclude that $\kappa$ is chordal.

(ii)$\Rightarrow$(iii)
Let $P_F$
  be the projection from $H$ onto $\ell^2(F)$ (when the latter is
  viewed as a subspace of $H$ in the natural way).  If $T\in \cl
  S(\kappa)$ is a positive operator then $T = \lim_F P_F T P_F$ in the
  weak* topology. On the other hand, $\kappa_F$ is
  chordal and hence, by~\cite{pps}, $P_F T P_F$ is the sum of rank
  one positive operators in $\cl S(\kappa)$.

  (iii)$\Rightarrow$(i) 
  Suppose that $\psi: \kappa \to \cl B(K)$ is a
  partially positive Schur multiplier. It is clear that $\psi_F:=
  \psi|_{\kappa_F}$ is a partially positive Schur multiplier
  on~$\kappa_F$.
  Since $\cl S(\kappa_F) = P_F \cl S(\kappa) P_F\subseteq \cl S(\kappa)$,
every positive operator in $\cl S(\kappa_F)$ is the weak* limit of
sums of rank one positive operators in $\cl S(\kappa_F)$;
since $\cl S(\kappa_F)$ is finite dimensional, we have that, in fact,  
every positive operator in $\cl S(\kappa_F)$ is the norm limit of 
sums of rank one positive operators in $\cl S(\kappa_F)$. 
Now \cite{pps}
implies that there exists a positive Schur multiplier  $\nph_F : F\times F \to \cl B(K)$  
whose restriction to $\kappa_F$ coincides with $\psi_F$.
Let $\tilde{\nph}_F: X \times X \to \cl B(K)$ be defined by
  \[
  \tilde{\nph}_F(x,y)=
    \begin{cases}
      \nph_F(x,y) & \text{if $x,y\in F$;} \\
      0 & \hbox{otherwise.}
    \end{cases}
  \]
  Then the map $S_{\tilde{\nph}_F}: \cl B(H) \to \cl B(H\otimes K)$ is
  completely positive.  On the other hand, since $\psi$ is a Schur
  multiplier on~$\kappa$ and~$I\in \cl S(\kappa)$, by \cite[Proposition 3.6]{Pa}, we have that
  \begin{align*}
    \|S_{\tilde{\nph}_F}\|_{\cb} &= \|S_{\tilde{\nph}_F}(I)\|
    = \max_{x\in F} \|\nph_F(x,x)\| = \max_{x\in F} \|\psi_F(x,x)\|
    \\&\leq \sup_{x\in X}\|\psi(x,x)\|=\|S_\psi(I)\|<\infty.
  \end{align*}
  So $\{S_{\tilde{\nph}_F}\}_F$ is a uniformly bounded family of completely
  positive maps from $\cl B(H)$ into $\cl B(H\otimes K)$.  
By
  \cite[Theorem~7.4] {Pa}, there exist 
a subnet $(\Phi_{F'})_{F'}$ and 
  a completely positive map 
  $\Phi :  \cl B(H) \to \cl B(H \otimes K)$ such that $\Phi_{F'}(T)\to \Phi(T)$ along $F'$
  in the weak* topology, for every $T\in \cl B(H)$. 
We have that $\Phi(ATB) = (A\otimes I) \Phi(T)(B\otimes I)$ for all
diagonal operators $A$, $B$ and all $T\in \cl B(H)$, and this easily implies 
that $\Phi(E_{x,y}) = E_{x,y}\otimes \nph(x,y)$ for some
$\nph(x,y)\in \cl B(K)$. 
Since $\Phi_{F'}(E_{x,y}) = E_{x,y}\otimes \psi(x,y)$ for $(x,y)\in \kappa_{F'}$, 
the map $\nph : X\times X\to \cl B(K)$ extends $\psi$.
Now, for $T = (t_{x,y})_{x,y\in X} \in \cl B(H)$, $\xi,\eta\in K$ and $x,y\in X$, we have 
\begin{eqnarray*}
(\Phi(T)(e_y\otimes\xi),e_x\otimes \eta) 
& = & 
((E_{x,x}\otimes I)\Phi(T)(E_{y,y}\otimes I)(e_y\otimes\xi),e_x\otimes \eta)\\
& = & 
(\Phi(t_{x,y}E_{x,y})(e_y\otimes\xi),e_x\otimes \eta)
= (t_{x,y}\nph(x,y)\xi,\eta),
\end{eqnarray*}
so $\Phi(T) = (t_{x,y}\nph(x,y))_{x,y}$. 
It follows that $\nph$ is a Schur multiplier and $\Phi = S_{\nph}$; since $\Phi$ is completely positive, 
$\nph$ is positive.
\end{proof}

\section{Positivity domains}\label{s_pdom}

In this section, we study the domains of Schur multipliers over 
arbitrary standard $\sigma$-finite measure spaces.
To set the stage, we recall some measure theoretic terminology~\cite{eks}.
Let $(X,\mu)$ be a standard $\sigma$-finite measure space.  
The set $X\times X$ will be equipped with the product $\sigma$-algebra
and the product measure $\mu\times\mu$.
A subset
$E\subseteq X\times X$ is called \emph{marginally null} if $E\subseteq
(M\times X)\cup (X\times M)$, where $M\subseteq X$ is null.  We call
two subsets $E,F\subseteq X\times X$ {\it marginally equivalent} 
(resp. {\it  equivalent}),
and write $E\cong F$ (resp. $E\sim F$), 
if their symmetric difference is marginally null (resp. null with respect to 
product measure).
We say that $E$ is \emph{marginally contained} in $F$ (and write
$E\subseteq_{\omega} F$) if the set difference $E\setminus F$ is
marginally null.  A subset $\kappa\subseteq X\times X$ will be called
\begin{itemize}
\item a \emph{rectangle} if $\kappa = \alpha\times\beta$ where
$\alpha,\beta$ are measurable subsets of~$X$;
\item a \emph{square} if $\kappa = \alpha\times\alpha$ where
$\alpha$ is a measurable subset of~$X$;
\item {\it $\omega$-open} if it is marginally equivalent to a countable union of rectangles, and 
\item {\it $\omega$-closed} if its complement $\kappa^c$ is $\omega$-open.
\end{itemize}
Recall that, by~\cite{stt_cl}, if $\cl E$ is any collection of 
$\omega$-open sets, then there exists a smallest, up to marginal
equivalence, $\omega$-open set $\cup_{\omega}\cl E$, called the
\emph{$\omega$-union} of $\cl E$, such that every
set in~$\cl E$ is marginally contained in $\cup_{\omega}\cl E$.
Given a measurable set $\kappa$, one defines its
\emph{$\omega$-interior} to be
\[\ointer(\kappa) =
\cup_{\omega}\{R : R \, \mbox{ is a rectangle with } R \subseteq \kappa\}.\]
The \emph{$\omega$-closure} $\ocl(\kappa)$ of~$\kappa$ is defined as
the complement of $\ointer(\kappa^c)$. 
For a set $\kappa\subseteq X\times X$, we write $\hat{\kappa} =
\{(x,y)\in X\times X : (y,x)\in \kappa\}$.

Unless otherwise stated, we use the symbol $H$ to denote the Hilbert
space $L^2(X,\mu)$. For each $a\in L^{\infty}(X,\mu)$, let $M_a :
H\to H$ be the multiplication operator given by $M_a f = af$ and let $\cl
D = \{M_a : a\in L^{\infty}(X,\mu)\}$ be the 
multiplication algebra of $L^{\infty}(X,\mu)$; we have that $\cl D$ is a masa in $\cl B(H)$. 
For a measurable subset
$\alpha\subseteq X$, we let $\chi_{\alpha}$ denote the characteristic function of $\alpha$,
and set $P(\alpha) = M_{\chi_{\alpha}}$, a projection in~$\D$.
A \emph{$\cl D$-bimodule} (or simply a \emph{masa-bimodule}) 
is a subspace $\cl S\subseteq \cl B(H)$ such
that $\cl D\cl S\cl D\subseteq \cl S$. 

Let~$\kappa$ be a measurable subset of~$X\times X$. An operator $T\in
\cl B(H)$ is said to be \emph{supported by} $\kappa$ if $P(\beta)TP(\alpha) =
0$ whenever $(\alpha\times\beta)\cap \kappa \cong \emptyset$.  Given
any masa-bimodule $\cl U$, there exists a unique (up to marginal
equivalence) smallest $\omega$-closed set $\kappa\subseteq X\times X$
such that every operator in $\cl U$ is supported by
$\kappa$~\cite{eks}.  The set~$\kappa$ is denoted by $\supp \cl U$ 
and is called the \emph{support} of $\cl U$.

For $k\in L^2(X\times X)$, the Hilbert-Schmidt operator $T_k : H\to H$
with integral kernel $k$ is defined by
\[T_k f(y) = \int_X k(y,x)f(x) d\mu(x), \ \ f \in H, y\in X.\]
For any measurable, $\omega$-closed subset 
$\kappa\subseteq X\times X$, let 
$$\cl S(\kappa) = \overline{\{T_k :  k \in L^2(\kappa)\}}^{w^*},$$
where $L^2(\kappa)$ is the space of functions in $L^2(X\times X)$ which are supported
on $\kappa$. It is easy to see that every operator $T\in \cl S(\kappa)$ is supported by 
$\hat \kappa$, so $\supp \cl S(\kappa)\subseteq_{\omega} \hat \kappa$.
Note that the latter inclusion may be strict. 
Indeed, if $\kappa$ has product measure zero then $\cl S(\kappa) = \{0\}$ 
and hence $\supp \cl S(\kappa)\cong \emptyset$; however, 
$\kappa$ does not need to be marginally null.

Suppose that $X$ is equipped with the counting measure.  Then $\cl D$
is the algebra of all diagonal operators on $\ell^2(X)$ and $\cl
S(\kappa)$ is the weak* closure of the linear span of the matrix units
$E_{x,y}$, with $(x,y)\in \kappa$; it thus coincides with the space
defined in Section \ref{s_Bell2} (see (\ref{eq_dis})). In particular, $\cl S(\kappa)$ is
generated, as a weak* closed subspace, by the rank one operators it contains. 
In view of Proposition \ref{th_ro} below, in this general context it is therefore 
natural to restrict attention to 
sets $\kappa$ which contain ``plenty of rectangles''. We will now make
this intuitive idea precise.

\begin{definition}
  A measurable subset~$\kappa\subseteq X\times X$ is said to be:

(i) \  \emph{generated by rectangles} if $\kappa\cong\ocl(\ointer(\kappa))$;

(ii) \emph{symmetric} if~$\kappa\cong\hat\kappa$.
\end{definition}

For $\xi,\eta\in H$, we denote by $\xi\otimes\eta^*$ the rank one operator on $H$
given by $(\xi\otimes\eta^*) (\zeta) = (\zeta,\eta)\xi$. 
We denote by $\cl S_1(\kappa)$ the set of all rank one operators in~$\cl  S(\kappa)$.

\begin{proposition}\label{th_ro}
  Let $\kappa\subseteq X\times X$ be an $\omega$-closed set.

(i) \  If~$\xi,\eta\in H$, then the rank one operator~$\xi\otimes
  \eta^*$ belongs to $\cl S(\kappa)$ if and only if $(\supp \xi)\times
  (\supp \eta)$ is marginally contained in~$\kappa$.

(ii) The set $\cl S_1(\kappa)$ generates~$\cl S(\kappa)$ as a weak* closed
  subspace of~$\cl B(H)$ if and only if~$\kappa$ is equivalent to $\ocl(\ointer(\kappa))$. 
\end{proposition}

\begin{proof}
(i) Suppose that $\xi\otimes\eta^* \in \cl S(\kappa)$. Then the operator $\xi\otimes\eta^*$
is supported by $\hat{\kappa}$. Writing 
$\alpha = \supp\xi$, $\beta = \supp\eta$ and 
$\kappa^c \cong \cup_{i=1}^{\infty} \alpha_i\times\beta_i$, 
where $\alpha_i,\beta_i\subseteq X$ are measurable for $i\in \bb{N}$, 
we have, in particular, that 
$(\beta_i\times\alpha_i)\cap \hat{\kappa} \cong \emptyset$,
so 
$$(\chi_{\alpha_i}\xi)\otimes (\chi_{\beta_i}\eta)^* = P(\alpha_i)(\xi\otimes\eta^*)P(\beta_i) = 0,$$
hence either 
$\alpha_i\cap \alpha$ or $\beta_i\cap \beta$ is null, for each $i\in \bb{N}$. 
It follows that $(\alpha\times\beta)\cap \kappa^c \cong \emptyset$, so 
$\alpha\times\beta\subseteq_{\omega}\kappa$.

Conversely, if $\alpha\times\beta\subseteq_{\omega}\kappa$ then 
the integral kernel of $\xi\otimes\eta^*$ is clearly in $L^2(\kappa)$, and thus
$\xi\otimes\eta^*\in \cl S(\kappa)$.

(ii) 
Set $\lambda = \ocl(\ointer(\kappa))$. 
Let $\cl V = \overline{{\rm span}(\cl S_1(\lambda))}^{w^*}$ and 
$\cl U = \Ref(\cl V)$ be the reflexive hull of $\cl V$ in the sense of \cite{ls}. 
By \cite[Theorem 5.2]{eks}, $\supp\cl U \cong \hat{\lambda}$; 
since $\supp\cl V \cong \supp \cl U$ \cite{eks}, 
we conclude that $\supp\cl V \cong \hat{\lambda}$. 
Thus, 
$$\cl V\subseteq \cl S(\lambda)\subseteq \cl U.$$
Since Hilbert-Schmidt operators are pseudo-integral in the sense of \cite{a}, 
we have that both $\cl V$ and $\cl S(\lambda)$ are masa-bimodules,
generated as weak* closed subspaces by pseudo-integral operators supported by $\hat{\lambda}$. 
It follows from \cite{a} 
that $\cl V = \cl S(\lambda)$. 

If $\kappa\sim \lambda$ then 
$$\cl S(\kappa) = \cl S(\lambda) = \cl V = \overline{{\rm span}(\cl S_1(\kappa))}^{w^*}.$$
Conversely, suppose that $\overline{{\rm span}(\cl S_1(\kappa))}^{w^*} = \cl S(\kappa)$. 
By (i), $\cl S_1(\lambda) = \cl S_1(\kappa)$, and the previous paragraph implies that
$$\overline{{\rm span}(\cl S_1(\kappa))}^{w^*}
= \overline{{\rm span}(\cl S_1(\lambda))}^{w^*} = \cl S(\lambda).$$
Thus, $\cl S(\kappa) = \cl S(\lambda)$ and it now easily follows that $\kappa\sim\lambda$. 
\end{proof}

Let~$\Delta = \{(x,x) : x\in X\}$ denote the diagonal of~$X$.

\begin{proposition}\label{p_opsc}
  If $\kappa\subseteq X\times X$ is generated by rectangles, then the following are equivalent:
  \begin{enumerate}[(i)]
  \item $\cl S(\kappa)$ is an operator system;
  \item $\kappa$ is symmetric and~$\Delta\subseteq_\omega\kappa$.
  \end{enumerate}
\end{proposition}
\begin{proof} 
(i)$\Rightarrow$(ii)
By Proposition \ref{th_ro} and its proof, 
$\cl S(\kappa) = \overline{{\rm span}(\cl S_1(\kappa))}^{w^*}$ and
$\supp \cl S(\kappa) \cong \hat{\kappa}$. Similarly, 
letting $\cl S(\kappa)^* = \{T^* : T\in \cl S(\kappa)\}$, we have that 
$\supp \cl S(\kappa)^* \cong \kappa$; thus $\hat{\kappa}\cong \kappa$. 
On the other hand, by assumption, $I\in \cl S(\kappa)$; since $\cl S(\kappa)$ is a 
masa-bimodule, we have that $\cl D\subseteq \cl S(\kappa)$. Thus, 
$\Delta \cong \supp\cl D \subseteq_{\omega} \kappa$.

(ii)$\Rightarrow$(i)
If $k\in L^2(\kappa)$ then the function $k^*$ given by $k^*(x,y) = \overline{k(y,x)}$
is in $L^2(\hat{\kappa})$. Since $\kappa\cong \hat{\kappa}$, we have that 
$k^*\in L^2(\kappa)$ and thus $T_k^* = T_{k^*}\in \cl S(\kappa)$. 
Hence, $\cl S(\kappa)^* = \cl S(\kappa)$. 

Since $\cl D$ consists of pseudo-integral operators in terms of \cite{a} and 
$\supp\cl D \cong \Delta$, 
by Proposition \ref{th_ro} and its proof we have that 
$\cl D\subseteq \cl S(\kappa).$
\end{proof}

\begin{definition}\label{d_pd}
  A measurable set $\kappa\subseteq X\times X$ will be called a
  \emph{positivity domain} if $\kappa$ is generated by rectangles,
  $\kappa$ is symmetric and $\Delta \subseteq_{\omega} \kappa$.
\end{definition}

By Proposition~\ref{th_ro} and Proposition~\ref{p_opsc}, an $\omega$-closed set 
$\kappa\subseteq X\times X$ is equivalent to a positivity domain if and only if $\cl S(\kappa)$
is an operator system, which is generated (as a weak* closed linear space) by the 
rank one operators it contains. Note that every operator system is generated, as a linear space,
by the positive operators it contains. 
However, as we will show 
in Corollary \ref{c_exi}, $\cl S(\kappa)$ does not need to contain 
a \emph{positive rank one} operator. 
It is worth noting that 
this phenomenon does not occur in the case of a discrete measure space $X$;
indeed, in this case, any diagonal matrix unit belongs to $\cl S(\kappa)$.

\begin{remark}\label{rk_36} 
Recall that a \emph{Schur idempotent} is a weak* continuous $\cl
D$-bimodule map $\Phi\in \B(\B(H))$ such that $\Phi\circ\Phi = \Phi$.
  Suppose that~$\Phi$ is a (completely) positive Schur idempotent; in this
  case, the range $\ran \Phi$ of $\Phi$ is selfadjoint.
  Trivially, $\ran \Phi$ is an operator system if and only if
  $\Phi$ is unital, that is, if and only if $\Phi(I) = I$.  
By \cite[Proposition 3.6]{Pa}, $\|\Phi\| = \|\Phi(I)\|$, and it follows that if $\ran \Phi$ is an operator system 
then $\Phi$ is contractive. 
Conversely, if $\|\Phi\| = 1$ then, by~\cite{kp}, $\ran \Phi = \cl S(\kappa)$, where $\kappa
  \cong \bigcup_{i=1}^{\infty}\alpha_i\times\alpha_i$ for some sequence
  $(\alpha_i)_{i=1}^{\infty}$ of pairwise disjoint measurable subsets
  of $X$ with $\cup_{i=1}^{\infty}\alpha_i = X$, and hence $\cl S(\kappa)$ is in this case an operator system.
\end{remark}

We next introduce a way to quantify the amount of positive rank one 
operators contained in $\cl S(\kappa)$, where $\kappa$ is a positivity domain. 
To this end, we define the
\emph{square interior} $\sqinter(\kappa)$ of~$\kappa$ by
  \[ \sqinter(\kappa) \stackrel{def}{=} \bigcup\mbox{}_\omega\{Q : Q \, \mbox{ is a square with }Q\subseteq
    \kappa\}.\] 
By Proposition \ref{th_ro} (i), 
$\cl S(\kappa)$ contains a positive rank one operator if and only if 
$\sqinter(\kappa) \not\cong \emptyset$. 
We will say that~$\kappa$ is \emph{generated by squares} if~$\kappa \cong \ocl(\sqinter(\kappa))$. 
Theorem \ref{th_charomi} below characterises the positivity domains with this property. 
  For a positivity domain $\kappa\subseteq X\times X$, let
  $\S_1^+(\kappa)$ denote the set of positive rank one operators
  in~$\S(\kappa)$, and set
  \[[\cl S_1^+(\kappa)] =
\left\{\sum_{i=1}^k R_i :\text{$k\in\bb N$ and $R_i\in
        \S_1^+(\kappa)$ for $1\leq i\leq k$}\right\}.
   \]

\begin{theorem}\label{th_charomi}
Let $\kappa\subseteq X\times X$ be a positivity domain. 
The following are equivalent:

\begin{enumerate}[(i)]
\item there is a unital contractive Schur idempotent $\Phi$ with
  $\ran\Phi\subseteq \cl S(\kappa)$;
\item there is a countable partition $X=\bigcup_{i=1}^\infty \alpha_i$
  of~$X$ into measurable subsets~$\alpha_i$ so that
  $\alpha_i\times\alpha_i\subseteq_{\omega} \kappa$ for each $i\in
  \bb{N}$;

  \item $I\in \overline{[\cl S_1^+(\kappa)]}^{w^*}$;

  \item $\Delta\subseteq_{\omega} \sqinter(\kappa)$;
  \item $\kappa$ is generated by squares.
\end{enumerate}
\end{theorem}
\begin{proof}
  (i)$\Leftrightarrow$(ii) follows from~\cite{kp} (see also
  Remark~\ref{rk_36}), while the implication (ii)$\Rightarrow$(iv) 
  follows from the inclusions 
  $$\Delta\subseteq \bigcup_{i=1}^{\infty} \alpha_i\times\alpha_i \subseteq_{\omega} 
\sqinter(\kappa).$$

  (ii)$\Rightarrow$(iii) 
    Clearly, $P(\alpha_j)\in \overline{[\cl S_1^+(\alpha_j\times\alpha_j)]}^{w^*}$; thus
  $P(\alpha_j)\in \overline{[\cl S_1^+(\kappa)]}^{w^*}$ and hence
  \[I = \mbox{w}^*\mbox{-}\lim_{N\to\infty} \sum_{j=1}^N P(\alpha_j)\in 
  \overline{[\cl S_1^+(\kappa)]}^{w^*}.\]

  (iii)$\Rightarrow$(i) Let $\Sigma$ be the collection of
  all countable sets $\{\Psi_i\}_{i\in \bb{N}}$ where each $\Psi_i$ is a
  contractive Schur idempotent with $\ran\Psi_i\subseteq \cl
  S(\kappa)$ and $\Psi_i\Psi_j = 0$ for all $i\neq j$. We order
  $\Sigma$ by inclusion.  The set $\Sigma$ is non-empty, since
  it contains the singleton consisting of the zero map.
  Suppose that $\Lambda\subseteq \Sigma$ is a non-empty chain. Since
  $H$ is separable, $\Lambda$ is countable, and hence its union is its
  upper bound.  By Zorn's Lemma, there exists a maximal element in
  $\Sigma$, say $\{\Phi_i\}_{i\in \bb{N}}$. Define $\Phi$ by $\Phi(T) =
  \sum_{i=1}^{\infty} \Phi_i(T)$, $T\in \cl B(H)$ (where the sum
  converges in the strong operator topology), so that~$\Phi$ is a
  contractive Schur idempotent with range contained
  in~$\S(\kappa)$. It remains to show that $\Phi$ is unital, or
  equivalently, that~$I\in \ran(\Phi)$. Let $(\alpha_i)_{i\in \bb{N}}$
  be a (countable) family of mutually disjoint measurable subsets of
  $X$ such that $\Phi(T) = \sum_{i=1}^{\infty} P(\alpha_i) T
  P(\alpha_i)$, $T\in \cl B(H)$ (see~\cite{kp} or Remark~\ref{rk_36}).
  Set $P = \sum_{i=1}^{\infty} P(\alpha_i)$ (where the sum converges
  in the strong operator topology).  We claim that $P = I$. To see
  this, suppose that $P^{\perp} = P(\alpha)$ for some measurable subset
  $\alpha$ of $X$ and note that $P^{\perp}\cl S(\kappa) P^{\perp} = \cl
  S((\alpha\times\alpha)\cap \kappa)$ is an operator system on the Hilbert space
  $P^{\perp}H$.  Since $\cl S(\kappa)$ satisfies the assumption in
  (iii), so does $P^{\perp}\cl S(\kappa) P^{\perp}$; so
  if~$P^\perp\ne0$, then $\S_1^+((\alpha\times\alpha)\cap \kappa)\ne \emptyset$. Hence by
  Proposition~\ref{th_ro}, there exists a non-trivial square $\beta\times\beta$,
  marginally contained in $(\alpha\times\alpha)\cap \kappa$.  Letting $\Phi_0$ be
  the Schur idempotent given by $\Phi_0(T) = P(\beta)TP(\beta)$, we
  see that the family $\{\Phi_i\}_{i=0}^{\infty}$ strictly contains
  $\{\Phi_i\}_{i=1}^{\infty}$, a contradiction. It follows that $P = I$
  and hence $I\in \ran\Phi$.

  (iv)$\Rightarrow$(ii) Suppose that $\Delta\subseteq_{\omega}
  \sqinter(\kappa)$. By \cite[Lemma~2.1]{stt_cl} there exists a family $(\beta_i)_{i\in
    \bb{N}}$ of measurable (not necessarily mutually disjoint) subsets
  of $X$ such that
  \[\Delta \subseteq_{\omega}
  \bigcup_{i=1}^{\infty} \beta_i\times \beta_i \subseteq_{\omega} \kappa.\] This
  implies that, up to a null set, $\bigcup_{i=1}^{\infty} \beta_i =
  X$.  Let $\alpha_1 = \beta_1$ and $\alpha_n = \beta_n \setminus
  (\bigcup_{j=1}^{n-1} \alpha_j)$.  Then $\alpha_n\times\alpha_n\subseteq
  \beta_n\times \beta_n\subseteq_{\omega} \kappa$ and $\bigcup_{n=1}^{\infty}
  \alpha_n = X$, up to a null set which we may adjoin to
  (say)~$\alpha_1$.

  (ii)$\Rightarrow$(v) Let~$(\alpha_i)_{i\in \bN}$ be the partition
  from~(ii). Given a rectangle $\beta\times \gamma\subseteq_\omega
  \kappa$, observe that 
  $\gamma\times \beta\subseteq_\omega \hat \kappa  \cong \kappa$ since~$\kappa$ is
  symmetric, and for~$i,j\in\bb N$ let $\beta_i=\beta\cap
  \alpha_i$ and $\gamma_j=\gamma\cap \alpha_j$. Then
  \begin{align*}
\beta_i\times\gamma_j    &\subseteq
    (\beta_i\cup \gamma_j)\times (\beta_i\cup\gamma_j)\\
    & = (\beta_i\times \beta_i)\cup (\beta_i\times
    \gamma_j)\cup(\gamma_j\times \beta_i)\cup
    (\gamma_j\times \gamma_j)\\&\subseteq (\alpha_i\times\alpha_i)\cup(\beta\times
    \gamma)\cup(\gamma\times\beta)\cup (\alpha_j\times\alpha_j) \subseteq_\omega
    \kappa.
  \end{align*}
  It follows that 
  $$\beta\times\gamma = \bigcup_{i,j = 1}^{\infty} \beta_i\times\gamma_j\subseteq_\omega \kappa.$$
  Taking the $\omega$-union over all such rectangles, we conclude that
  \[ \ointer(\kappa)\cong \sqinter(\kappa).\] 
  Since~$\kappa$ is generated by rectangles, we have
  \[\kappa\cong\ocl(\ointer(\kappa))\cong
  \ocl(\sqinter(\kappa)),\]
  and hence $\kappa$ is generated by squares.

  (v)$\Rightarrow$(iv) 
 By \cite[Lemma~2.1]{stt_cl}, we can write  $\sqinter(\kappa) \cong \cup_{i=1}^{\infty} (\alpha_i\times\alpha_i)$, 
 for some 
 sequence $(\alpha_i)_{i\in \bb{N}}$ of measurable subsets of $X$. 
 Set $Y = \cup_{i\in \bb{N}}\alpha_i$. Then $\sqinter(\kappa)\subseteq Y\times Y$ and, 
 since $Y\times Y$ is $\omega$-closed, 
 $$\kappa \cong \ocl(\sqinter(\kappa))\subseteq Y\times Y.$$ 
 Since $\Delta \subseteq_{\omega}\kappa$, 
 we have $Y \sim X$. It follows that $\Delta \subseteq_{\omega} \sqinter(\kappa)$.
 \end{proof}

\noindent {\bf Remark. } It is easy to see that 
if $\kappa\subseteq X\times X$ is a positivity domain and $\lambda$ $=$ $\ocl(\sqinter(\kappa))$ 
then the operator system $\cl S(\lambda)$ coincides with 
the weak* closure of the linear span of $[\cl S_1^+(\kappa)]$.
Thus, the square interior of $\kappa$ can be viewed as a 
measure of the quantity of the positive rank one operators in $\cl S(\kappa)$.

\section{Partially defined Schur multipliers}\label{s_e}

Let $(X,\mu)$ be a standard measure space. 
Recall that a function $\nph \in L^{\infty}(X\times X)$
is called a \emph{Schur multiplier} if the map $S_{\nph}$ defined on
the space of all Hilbert-Schmidt operators on $H = L^2(X,\mu)$ by
\begin{equation}\label{eq_nphk}
S_{\nph}(T_k) = T_{\nph  k}, \ \ \ k\in L^2(X\times X),
\end{equation}
is bounded in the operator norm.

If $\nph$ is a Schur multiplier, we denote again by $S_{\nph}$ the
bounded weak* continuous extension of this map to $\cl B(H)$; note
that $S_{\nph}$ is a \emph{masa-bimodule map} in the sense that
$S_{\nph}(BTA) = BS_{\nph}(T)A$, for all $T\in \cl B(H)$, $A,B\in \cl
D$. By~\cite[Theorem 2.1]{smith}, $S_{\nph}$ is automatically completely bounded.

We denote the set of all Schur multipliers on $X\times X$ by
$\frak{S}(X)$.  A Schur multiplier $\nph$ is called
\emph{positive} if the map $S_{\nph} : \cl B(H)\to \cl B(H)$ is
positive. In this case, $S_{\nph}$ is automatically completely
positive (see~\cite[Theorem 2.1]{smith} and~\cite[Lemma 4.3]{naomi_tt}). We write
$\frak S(X)^+$ for the set of positive Schur multipliers on~$X\times
X$.

We record the following well-known fact, which follows from results of
Haagerup~\cite{haag} and Smith~\cite{smith}.

\begin{theorem}\label{th_schur}
  Let $\Phi : \cl B(H) \to \cl B(H)$ be a linear map. The following are equivalent:
  \begin{enumerate}[(i)]
  \item  $\Phi$ is a bounded weak* continuous $\cl D$-bimodule
    map;
  \item  $\Phi$ is a completely bounded weak* continuous $\cl
    D$-bimodule map;
  \item  $\Phi = S_{\nph}$ for some $\nph \in \frak{S}(X)$;
  \item  there exist families $(a_i)_{i\in \bb{N}}, (b_i)_{i\in \bb{N}}\subseteq L^{\infty}(X)$
  such that 
  $$\esssup_{x\in X} \sum_{i=1}^{\infty} |a_i(x)|^2 < \infty, \ \ 
  \esssup_{y\in X} \sum_{i=1}^{\infty} |b_i(y)|^2 < \infty$$ 
  and 
  $$\nph(x,y) = \sum_{i=1}^{\infty} a_i(x) b_i(y), \ \ \mbox{ a.e. } (x,y)\in X\times X.$$
  \end{enumerate}
\end{theorem}

We extend the definition of a Schur multiplier given above
to functions defined on proper subsets of $X\times X$. 
Let $\kappa\subseteq X\times X$ be a measurable set, equipped with the
induced $\sigma$-algebra and, as before, identify
$L^2(\kappa)$ with the space of all functions in $L^2(X\times X)$
supported on $\kappa$.

\begin{definition}\label{d_sch}
  Let $\kappa\subseteq X\times X$ be a measurable subset generated by rectangles.
  A measurable function $\nph : \kappa \to \bb{C}$ will be called a
  \emph{Schur multiplier} if there exists $C > 0$ such that
  $\|T_{\nph k}\| \leq C \|T_{k}\|$, for every $k\in
  L^2(\kappa)$.
\end{definition}

Let 
$$\cl S_0(\kappa) = \overline{\{T_k :  k \in L^2(\kappa)\}}^{\|\cdot\|};$$ 
thus, $\cl S_0(\kappa)\subseteq \cl S(\kappa) \cap \cl K(H)$ and 
$\overline{\cl S_0(\kappa)}^{w^*} = \cl S(\kappa)$. 
If $\kappa$ is a positivity domain, we equip 
$\cl S_0(\kappa)$ with the structure of a matrix ordered space (see {\it e.g.} \cite{werner}), 
arising from its inclusion into the operator system $\cl S(\kappa)$;
thus, it makes sense to talk about positive or 
completely positive maps on $\cl S_0(\kappa)$.

For functions $\nph,\psi : \kappa\to \bb{C}$, we write $\nph\sim\psi$
if $\{(x,y) : \nph(x,y)\neq \psi(x,y)\}$ is a null set.

\begin{proposition}\label{p_chsch}
Let $\kappa\subseteq X\times X$ be generated by rectangles and let
$\nph : \kappa\to \bb{C}$ be a  measurable function. The following are equivalent:
\begin{enumerate}[(i)]
\item $\nph$ is a Schur multiplier;

\item there exists a Schur multiplier $\psi : X\times X\to \bb{C}$
  such that $\psi|_{\kappa} \sim \nph$;

\item there exists a unique completely bounded map 
$\Phi_0 : \cl S_0(\kappa) \to \cl S_0(\kappa)$ such that 
$\Phi_0(T_k) = T_{\nph k}$, for each $k\in L^2(\kappa)$;

\item there exists a unique completely bounded weak* continuous map
$\Phi : \cl S(\kappa) \to \cl S(\kappa)$ such that $\Phi(T_k) =
  T_{\nph k}$, for each $k\in L^2(\kappa)$.
\end{enumerate}
\end{proposition}
\begin{proof}
(i)$\Rightarrow$(ii)
Let 
$$\cl S_2(\kappa) = \{T_k : k \in L^2(\kappa)\}.$$
Since $\nph$ is a Schur multiplier, the map 
$\Phi_2 : \cl S_2(\kappa)\to \cl S_2(\kappa)$, given by 
$\Phi_2(T_k) = T_{\nph k}$, extends to a bounded linear map
$\Phi_0 : \cl S_0(\kappa)\to \cl S_0(\kappa)$. Moreover, since $\Phi_2$ is 
a $\cl D$-bimodule map, $\Phi_0$ is such as well. 
By Smith's theorem~\cite[Theorem 2.1]{smith}, $\Phi_0$ is completely bounded. 
By~\cite[Exercise 8.6 (ii)]{Pa}, there exists a completely bounded 
$\cl D$-bimodule map $\Psi_0 : \cl K(H)\to \cl B(H)$ such that 
$\Psi_0|_{\cl S_0(\kappa)} = \Phi_0$. Let 
$\Psi$ be the weak* continuous extension of $\Psi_0$ whose existence is guaranteed by Lemma \ref{r_exte} (iii);
we have that $\Psi$ is a completely bounded weak* continuous $\cl D$-bimodule map.
By Theorem~\ref{th_schur}, there exists 
$\psi\in \frak{S}(X)$ such that $\Psi = S_{\psi}$. 
For every $k\in L^2(\kappa)$ we have 
$$T_{\psi k} = S_{\psi}(T_k) = S_{\nph}(T_k) = T_{\nph k}.$$
It follows that 
$$\psi k \sim \nph k, \mbox{ for each } k\in L^2(\kappa),$$
and hence $\psi|_{\kappa}\sim \nph$.

(ii)$\Rightarrow$(iv) Take $\Phi = S_{\psi}|_{\cl S(\kappa)}$. 
The uniqueness of $\Phi$ follows from the fact that 
the Hilbert-Schmidt operators with integral kernels in $L^2(\kappa)$ 
are dense in $\cl S(\kappa)$.

(iv)$\Rightarrow$(iii)$\Rightarrow$(i) are trivial. 
\end{proof}

  If~$\nph\colon \kappa\to \mathbb C$ is a Schur multiplier, then we write
  $S_\nph\colon \cl S(\kappa)\to \cl S(\kappa)$ for the map~$\Phi$
  appearing in~(iv) above, and let $S_{\nph}^0$ be the restriction of $S_{\nph}$ to $\cl S_0(\kappa)$.

\begin{definition}\label{d_fp}
Let $\kappa\subseteq X\times X$ be a positivity domain.
A Schur multiplier $\nph : \kappa \to \bb{C}$ will be called \emph{partially positive}
if $\nph|_{\alpha\times\alpha}$ is a positive Schur multiplier
whenever $\alpha\subseteq X$ is a measurable set with 
$\alpha\times\alpha\subseteq \kappa$.
\end{definition}

We can characterise partial positivity in this context using rank one operators,
extending \cite[Lemma~4.2]{pps}, as follows.

\begin{proposition}\label{p_ppsch}
Let $\kappa$ be a positivity domain. A Schur multiplier
$\nph : \kappa\to \bb{C}$ is partially positive if and only if 
$S_{\nph}(\cl S_1^+(\kappa))\subseteq \cl B(H)^+$.
\end{proposition}
\begin{proof}
Suppose that $S_{\nph}(\cl S_1^+(\kappa))\subseteq \cl B(H)^+$
and that $\alpha\times\alpha\subseteq \kappa$ for a measurable set $\alpha\subseteq X$. 
If $T$ is a positive rank one operator supported by $\alpha\times\alpha$ then, 
by assumption, $S_{\nph}(T)\geq 0$. Since $S_{\nph}$ is weak* continuous and 
the weak* closed span of the positive rank one operators supported by $\alpha\times\alpha$
equals $\cl B(P(\alpha)H)^+$, we conclude that $\nph|_{\alpha\times\alpha}$ is a 
positive Schur multiplier.

Conversely, suppose that $\nph$ is partially positive and that $T\in \cl S(\kappa)$ is a positive 
rank one operator, say $T = \eta\otimes\eta^*$ for some $\eta\in H$. 
If $\supp\eta = \alpha$ then, by Proposition~\ref{th_ro}, $\alpha\times\alpha\subseteq_{\omega} \kappa$. 
By deleting a null set from $\alpha$, we may in fact suppose that 
$\alpha\times\alpha\subseteq \kappa$.
The assumption now implies that $S_{\nph}(T) \geq 0$. 
\end{proof}

We note that the main interest in Proposition \ref{p_ppsch} is when 
$\kappa$ is generated by squares; however, we formulate it in 
its present generality in view of Theorem \ref{th_posr}.

\begin{definition}\label{d_pc}
Let $\kappa\subseteq X\times X$ be a positivity domain and 
$\nph : \kappa\to \bb{C}$ be a Schur multiplier. 
We say that a measurable function $\psi : X\times X\to \bb{C}$ is 
a \emph{positive extension} of $\nph$ if $\psi$ is a positive Schur multiplier
and $\psi|_{\kappa}\sim \nph$.
\end{definition}

\noindent {\bf Remarks (i)} 
If a Schur multiplier $\nph : \kappa\to \bb{C}$ has a positive extension
then $\nph$ is necessarily partially positive. 

\smallskip

{\bf (ii)} 
Recall~\cite{eks} that a function $\nph : X\times X\to \bb{C}$ is called 
$\omega$-continuous if $\nph^{-1}(U)$ is an $\omega$-open set 
for every open subset $U\subseteq \bb{C}$. 
We note that if $\nph : \kappa\to \bb{C}$ has a positive extension, then
$\nph$ has an $\omega$-continuous positive extension. This follows from the fact that
if $\psi : X\times X\to \bb{C}$ is a positive Schur multiplier then there exists an 
$\omega$-continuous positive Schur multiplier $\psi' : X\times X\to \bb{C}$
such that $\psi$ and $\psi'$ differ on a set of product measure zero (see~\cite[Corollary 4.5]{naomi_tt}).

\begin{theorem}\label{th_ns}
Let $\kappa$ be a positivity domain. 
The following are equivalent, for a  partially positive Schur multiplier $\nph : \kappa\to \bb{C}$:
\begin{enumerate}[(i)]
\item $\nph$ has a positive extension;

\item the map $S_{\nph} : \cl S(\kappa) \to \cl S(\kappa)$ is positive;

\item the map $S_{\nph} : \cl S(\kappa) \to \cl S(\kappa)$ is completely positive;

\item the map $S^0_{\nph} : \cl S_0(\kappa) \to \cl S_0(\kappa)$ is positive;

\item the map $S^0_{\nph} : \cl S_0(\kappa) \to \cl S_0(\kappa)$ is completely positive.
\end{enumerate}
\end{theorem}
\begin{proof}
(i)$\Rightarrow$(ii) If $\psi$ is a positive extension of $\nph$ then 
$S_{\psi}$ is positive and hence so is its restriction to $\cl S(\kappa)$;
it is easily seen that this restriction coincides with $S_{\nph}$. 

(ii)$\Rightarrow$(iii) and 
(iv)$\Rightarrow$(v)
follow from the operator system version of 
R. R. Smith's theorem~\cite{smith} on automatic complete boundedness
(see~\cite[Lemma 4.3]{naomi_tt}). 

(iii)$\Rightarrow$(iv) is trivial.

(v)$\Rightarrow$(i) 
Let $\Phi = S_{\nph}^0$; thus, $\Phi$ is a completely positive linear map on $\cl S_0(\kappa)$. 
By \cite[Theorem 3.16 and Lemma 3.12]{r}, 
$\Phi$ can be extended to a completely positive 
map $\Phi_1$ on the operator system $\cl S_0(\kappa) + \bb{C}I$. 
By Arveson's Extension Theorem, 
there exists a completely positive map $\Psi_1 : \cl K(H) + \bb{C}I\to \cl B(H)$ extending $\Phi_1$. 
The restriction $\Psi$ of $\Psi_1$ to $\cl K(H)$ is then a completely positive 
extension of $\Phi$. 
Let $\tilde{\Psi} : \cl B(H) \to \cl B(H)$ be the completely positive weak* 
continuous extension of $\Psi$ whose existence is guaranteed by Lemma \ref{r_exte} (iii).
Let $\tilde{\Phi}$ be the restriction of $\tilde{\Psi}$ to $\cl S(\kappa)$; 
thus, $\tilde{\Phi}$ is a weak* continuous extension of $\Phi$, and since 
$\Phi$ is a $\cl D$-bimodule map, the same holds true for $\tilde{\Phi}$. 
We now have that $\tilde{\Psi}$ is a completely positive extension of 
$\tilde{\Phi}$; by~\cite[Exercise 7.4]{Pa}, $\tilde{\Psi}$ is a 
$\cl D$-bimodule map. 
By Theorem~\ref{th_schur}, there exists $\psi\in \frak{S}(X)$ such that 
$\tilde{\Psi} = S_{\psi}$; the function $\psi$ is the desired positive extension of $\nph$. 
\end{proof}

The next theorem is a measurable version of one of the main results of 
\cite{pps} concerning positive completions of partially positive matrices.
Recall that the projective tensor product 
\[\cl T(X) = L^2(X) \hat{\otimes } L^2(X)\]
can be canonically identified
with the predual of $\cl B(H)$. Indeed, each element $h\in
\cl T(X)$ can be written as a series $h = \sum_{i=1}^{\infty} f_i\otimes g_i$, 
where $\sum_{i=1}^{\infty} \|f_i\|_2^2 < \infty$
and $\sum_{i=1}^{\infty} \|g_i\|_2^2 < \infty$, and the duality pairing is then given
by 
\[\langle T,h\rangle = \sum_{i=1}^{\infty} (Tf_i,\overline{g_i}), \ \ T\in \cl B(H).\] 
It follows~\cite{a} that $h$ can be identified with a complex function
on $X\times X$, defined up to a marginally null set, and given by
$h(x,y) = \sum_{i=1}^{\infty} f_i(x)g_i(y)$. 
The positive cone $\cl T(X)^+$ consists of all 
functions $h\in \cl T(X)$ that define positive functionals on $\cl B(H)$,
that is, functions $h$ of the form 
$h = \sum_{i=1}^{\infty} f_i\otimes \overline{f_i}$,  where $\sum_{i=1}^{\infty} \|f_i\|_2^2 < \infty$.
It is well-known that a function
$\nph \in L^{\infty}(X\times X)$ is a Schur multiplier
if and only if $\nph h$ is equivalent (with respect to the product measure) to a
function contained in $\cl T(X)$ for every $h\in \cl T(X)$ (see \cite{peller}).

\begin{theorem}\label{th_posr}
Let $\kappa$ be a positivity domain. The following are equivalent:
\begin{enumerate}[(i)]
\item every partially positive Schur multiplier $\nph : \kappa\to \bb{C}$ has a positive extension;

\item $\cl S(\kappa)^+ = \overline{[\cl S_1^+(\kappa)]}^{w^*}$;

\item $\cl S_0(\kappa)^+ = \overline{[\cl S_1^+(\kappa)]}^{\|\cdot\|}$.
\end{enumerate}
\end{theorem}

\begin{proof}
Assume (i) holds. We establish 
(ii) and (iii) simultaneously.
It is clear that 
$\overline{[\cl S_1^+(\kappa)]}^{\|\cdot\|}\subseteq \cl S_0(\kappa)^+$. 
To show that the two cones are 
equal, it suffices by the Hahn-Banach Theorem to prove that if $h\in \cl T(X)$ is 
such that $\langle A,h\rangle \geq 0$ for all $A\in [\cl S_1^+(\kappa)]$, then 
$\langle A,h\rangle \geq 0$ for all $A\in \cl S(\kappa)^+$.

Thus, suppose that 
$$\langle A,h\rangle \geq 0, \ \ \ A\in [\cl S_1^+(\kappa)].$$
Then 
\[\langle \xi\otimes\xi^*, h\rangle \geq 0, \quad\text{whenever $\xi\in L^2(X)$ and $\xi\otimes\xi^* \in \cl S_1(\kappa)$,}\]
that is (see Proposition \ref{th_ro}),
$$\langle \xi\otimes\xi^*, h\rangle \geq 0, \ \ 
\mbox{whenever } (\supp\xi)\times (\supp\xi)\subseteq_{\omega} \kappa.$$

Suppose first that the measure $\mu$ is finite and $h\in \frak{S}(X)$; 
recall that $S_{h}$ denotes the corresponding bounded weak* continuous 
map on $\cl B(H)$, defined in (\ref{eq_nphk}).
For every $\eta\in L^2(X)$, we have 
$$(S_h(\xi\otimes\xi^*)\eta,\eta) = 
\int_{X\times X} h(x,y) \xi(y)\eta(y)\overline{\xi(x)\eta(x)} dx dy 
= \langle (\xi\eta)\otimes(\xi\eta)^*, h\rangle \geq 0$$
whenever $(\supp\xi)\times (\supp\xi)\subseteq_{\omega} \kappa.$
It follows that 
$S_h(\xi\otimes\xi^*)\geq 0$ for every rank one operator $\xi\otimes\xi^*$ that belongs to $\cl S(\kappa)^+$;
thus, $h|_{\kappa}$ is a partially positive Schur multiplier. 
By assumption, $h|_{\kappa}$ has an extension to a positive Schur multiplier $\tilde{h}$ on $X\times X$. 
We thus have
\begin{equation}\label{eq_htild}
(S_{\tilde{h}}(A)\chi_X,\chi_X) \geq 0, \ \ \ \ \mbox{for all } A\in \cl B(H)^+
\end{equation}
(note that $\chi_X \in L^2(X)$ since~$\mu$ is finite).
We claim that if $A\in \cl S(\kappa)$ then 
\begin{equation}\label{eq_tildht}
(S_{\tilde{h}}(A)\chi_X,\chi_X) = \langle A,h\rangle.
\end{equation}
Indeed, suppose that $k\in L^2(\kappa)$. Then 
\begin{eqnarray*}
(S_{\tilde{h}}(T_k)\chi_X,\chi_X) & = & \int_{X\times X} \tilde{h}(x,y) k(y,x) d\mu(x)d\mu(y)\\
& = & \int_{\kappa} h(x,y) k(y,x) d\mu(x)d\mu(y) = \langle T_k,h\rangle,
\end{eqnarray*}
so (\ref{eq_tildht}) holds if~$A$ is a Hilbert-Schmidt operator. 
Since $S_{\tilde{h}}$ is weak* continuous and 
the Hilbert-Schmidt operators are dense in $\cl S(\kappa)$, 
(\ref{eq_tildht}) is established.
It now follows that if $A\in \cl S(\kappa)^+$, then
$\langle A,h\rangle\geq 0$.

Next, relax the assumption that the measure $\mu$ be finite, but continue to suppose that 
$h\in \frak{S}(X)$. 
Let $(X_n)_{n\in \bb{N}}$ be an increasing sequence of meaurable sets of finite measure such that 
$X = \cup_{n\in \bb{N}} X_n$. 
By (\ref{eq_tildht}), 
\[(S_{\tilde{h}}(A)\chi_{X_n},\chi_{X_n}) = \langle P(X_n)AP(X_n),h\rangle \geq 0\]
whenever $A\in \cl S(\kappa)^+$. Passing to the limit as~$n\to \infty$, we conclude 
that $\langle A,h\rangle\geq 0$ whenever $A\in \cl S(\kappa)^+$. 

Finally, relax the assumption that $h$ be a Schur multiplier.  Write
$h = \sum_{i=1}^{\infty} f_i\otimes g_i$ where
$\sum_{i=1}^{\infty}\|f_i\|^2_2 < \infty$ and
$\sum_{i=1}^{\infty}\|g_i\|^2_2 < \infty$.  For each $N\in \bb{N}$,
let
$$X_N = \left\{x\in X : \sum_{i=1}^{\infty}|f_i(x)|^2 \leq N\right\}$$
and 
$$Y_N = \left\{y\in X : \sum_{i=1}^{\infty}|g_i(y)|^2 \leq N\right\}.$$
Then $X\setminus (\cup_{N\in \bb{N}} X_N)$ and $X\setminus (\cup_{N\in \bb{N}} Y_N)$
have measure zero. Let $Z_N = X_N\cap Y_N$, $N\in \bb{N}$. 
Then $X\setminus (\cup_{N\in \bb{N}} Z_N)$ has measure zero and, by Theorem \ref{th_schur},  
the restriction of $h$ to $Z_N\times Z_N$ is a Schur multiplier. 
By the previous paragraphs, 
$\langle P(Z_N)AP(Z_N),h\rangle\geq 0$ whenever $A\in \cl S(\kappa)^+$, 
and passing to the limit as~$N\to \infty$ shows that 
$\langle A,h\rangle \geq 0$ for all $A\in \cl S(\kappa)^+$.

(iii)$\Rightarrow$(i) and (ii)$\Rightarrow$(i)
follow from Proposition \ref{p_ppsch} and Theorem \ref{th_ns}.
\end{proof}

We note that, by Theorem~\ref{th_posr}, if a positivity domain admits positive extensions
of partially positive Schur multipliers, then it necessarily satisfies the equivalent conditions of 
Theorem~\ref{th_charomi}.

\section{Extending positive definite functions}\label{s_epdf}

In this section, we apply the results obtained previously to 
some questions arising in Abstract Harmonic Analysis. 
We assume, throughout the section, that $G$ is a 
locally compact second countable amenable group, 
unless $G$ is discrete, in which case no countability restriction is required. 
We denote by $m$ the left Haar measure on $G$.
We formulate sufficient conditions for unital symmetric 
subsets~$E\subseteq G$ which guarantee that every
positive definite function defined on $E$ can be extended to a positive definite
function defined on the whole group $G$.

Let $A(G)$ (resp. $B(G)$) be the Fourier (resp. the Fourier-Stieltjes) algebra of $G$ 
(see~\cite{eymard} for the definition and basic properties of these algebras).
Recall that $\cl T(G) = L^2(G)\hat{\otimes} L^2(G)$ 
and let $P : \cl T(G)\to A(G)$ be the contractive surjection given by 
$$P(\xi\otimes\eta) = \eta\ast \check{\xi}, \ \ \ \xi,\eta\in L^2(G)$$
(here $\ast$ denotes the usual convolution of functions and $\check\xi(s)=\xi(s^{-1})$, $s\in G$).
Let $A(G)^+$ (resp. $B(G)^+$) be the canonical positive cone of $A(G)$ 
(resp. $B(G)$). 
Recall that a function $f : G\to \bb{C}$ is called
\emph{positive definite} if for all finite subsets
$\{s_1,\dots,s_n\}\subseteq G$, the matrix $(f(s_i s_j^{-1}))_{i,j}$
is positive;
it is well-known that $B(G)^+$ coincides with 
the cone of all continuous positive definite functions on $G$.
More generally, if $E\subseteq G$ is a symmetric set (that is, $E^{-1} =
E$) containing the neutral element $e$ of $G$, we say that 
a function $f : E\rightarrow \bb{C}$ is positive definite if the matrix 
$(f(s_is_j^{-1}))_{i,j}$ is positive for any~$s_1,\dots,s_n\in G$ with
$s_is_j^{-1}\in E$ for $1\leq i,j\leq n$. 
We say that $f : G\to \bb{C}$ is positive definite on $E$ if $f|_E$ is 
positive definite. 

For a subset $E\subseteq G$, let 
$$E^* = \{(s,t)\in G\times G : ts^{-1}\in E\},$$ 
and for a function $f : E\to \bb{C}$, let $N(f) : E^*\to \bb{C}$
be given by 
$$N(f)(s,t) = f(ts^{-1}), \ \ \ \ (s,t)\in E^*.$$ 
It is well-known that 
$N$ maps $B(G)$ isometrically onto the set $\frak{S}_{\rm inv}(G)$ of all 
\emph{invariant} Schur multipliers, that is, the Schur multipliers 
$\nph : G\times G\to \bb{C}$ with the property that, for each $r\in G$, 
$\nph(sr,tr) = \nph(s,t)$ for marginally almost every $(s,t)\in G\times G$ (see \cite{bf}), 
and that the image of $B(G)^+$ under $N$ consists of the positive elements of 
$\frak{S}_{\rm inv}(G)$.

Since the sets of the form $E^*$ play an important role in 
positive definiteness, 
we first provide a description of when they are positivity domains (Theorem \ref{th_into}). 
We need an extra technical condition on the group $G$, known as Lebesgue density. 
If $G$ is a locally compact group equipped with a metric, 
we say that $t\in G$ is a point of density for a Borel set $\alpha\subseteq G$ 
if $\lim_{\epsilon\to 0} \frac{m(\alpha\cap B(t,\epsilon))}{m(B(t,\epsilon))} = 1$, 
where $B(t,\epsilon)$ is the open ball with centre $t$ and radius $\epsilon$. 
We say that $G$ satisfies the Lebesgue Density Theorem if there exists a left invariant metric on $G$
which induces its topology, 
such that for every Borel set $\alpha$, 
the subset
$$\alpha_0 = \{s\in \alpha : s \mbox{ a point of density of } \alpha\}$$ 
has full measure in $\alpha$. 
The classical Lebesgue Density Theorem states that the real line $\bb{R}$ satisfies 
this condition; it is also fulfilled for the groups $\bb{T}^n$, $\bb{R}^m$ and their products.

\begin{lemma}\label{l_ldt}
Let $G$ be a locally compact group that satisfies the Lebesgue Density Theorem, and let 
$E\subseteq G$ be a Borel set. Suppose that $\alpha$ and $\beta$ are non-null Borel subsets of $G$ 
such that $\beta\alpha^{-1}\subseteq E$. Let $\alpha_0$ (resp. $\beta_0$) be the 
subset of $\alpha$ (resp. $\beta$) consisting of its points of density. 
Then $\beta_0\alpha_0^{-1}\subseteq {\rm int}(E)$.
\end{lemma}
\begin{proof}
For $t\in G$, we have 
\begin{equation}\label{eq_Pcon}
P(\chi_{\alpha}\otimes\chi_{\beta})(t) = \int_G \chi_{\beta}(s)\check{\chi}_{\alpha}(s^{-1}t)ds 
= m(\beta\cap t\alpha).
\end{equation}
Let 
\begin{equation}\label{eq_U}
U := \{t\in G : m(\beta\cap t\alpha ) > 0\}.
\end{equation}
Since $P(\chi_{\alpha}\otimes\chi_{\beta})$ is a continuous function, (\ref{eq_Pcon}) implies that
$U$ is an open set. Note that $t\in \beta\alpha^{-1}$ if and only if $\beta\cap t\alpha\neq \emptyset$.
We show that 
\begin{equation}\label{eq_0}
\beta_0\alpha_0^{-1}\subseteq U.
\end{equation}
To this end, suppose that $t\in \beta_0\alpha_0^{-1}$, and write 
$t = yx^{-1}$, where $x\in \alpha_0$ and $y\in \beta_0$. Suppose that 
$m(\beta\cap t\alpha) = 0$; then $m(y^{-1}\beta\cap x^{-1}\alpha) = 0$. 
Thus, for every $\epsilon > 0$, 
$$m((y^{-1}\beta) \cap B(e,\epsilon)) + m((x^{-1}\alpha) \cap B(e,\epsilon)) \leq m(B(e,\epsilon))$$
and so 
$$2 = \lim_{\epsilon \to 0} \frac{m((y^{-1}\beta) \cap B(e,\epsilon))}{m(B(e,\epsilon))}
+ \lim_{\epsilon \to 0} \frac{m((x^{-1}\alpha) \cap B(e,\epsilon))}{m(B(e,\epsilon))} 
\leq 1,$$
a contradiction.
Thus, (\ref{eq_0}) holds true. 
On the other hand, clearly $U\subseteq E$; since $U$ is open, we have that 
$\beta_0\alpha_0^{-1}\subseteq {\rm int}(E)$.
\end{proof}

\begin{theorem}\label{th_into} 
  If $G$ is a locally compact group satisfying the Lebesgue Density
  Theorem and $E\subseteq G$ is a Borel subset of positive measure, then
\begin{equation}\label{eq_clinte}
  \ointer(E^*) \cong {\rm int}(E)^*\quad\text{and}\quad \ocl(E^*) \cong {\rm cl}(E)^*.
\end{equation}
  In particular, 
  
  (i) \ $E^*$ is a positivity domain if and only if $E$ is a symmetric set, $e\in E$, and 
  $E$ is the closure of its interior;
  
  (ii) $E^*$ is generated by squares if and only if $E^*$ 
  contains a non-null square, if and only if $E$ contains a symmetric open neighbourhood of $e$. 
\end{theorem}

\begin{proof}
  If
  $U\subseteq E$ is open then $U^*$ is an open, and hence an $\omega$-open,
  subset of $E^*$, and so ${\rm int}(E)^*\subseteq_{\omega} \ointer(E^*)$.
  Conversely, suppose that $\alpha\times\beta$ is a rectangle,
  marginally contained in $E^*$. After deleting sets of measure zero
  from $\alpha$ and $\beta$, we may assume that
  $\alpha\times\beta\subseteq E^*$. Let $\alpha_0$ (resp. $\beta_0$)
  be the set of all points of density in $\alpha$ (resp. $\beta$).  By assumption,
  $\alpha\setminus\alpha_0$ and $\beta\setminus\beta_0$ have measure
  zero.  By Lemma~\ref{l_ldt}, $(\beta_0\alpha_0^{-1})^*\subseteq
  {\rm int}(E)^*$, and since $\alpha_0\times\beta_0\subseteq
  (\beta_0\alpha_0^{-1})^*$, we have that
  $\alpha\times\beta\subseteq_{\omega} {\rm int}(E)^*$.  It follows
  that $\ointer(E^*) \cong {\rm int}(E)^*$. This easily implies that $\ocl(E^*)  \cong {\rm cl}(E)^*$. 
  
Statement (i) is immediate from identities (\ref{eq_clinte}) and the definition of a positivity domain.
To prove (ii), assume first that $E^*$ is generated by squares. Since $E^*$ is non-null, 
$E^*$ contains a non-null square, say $\alpha\times\alpha$. 
Let $U$ be the open set defined in (\ref{eq_U}), where we have taken $\beta = \alpha$. 
By the left invariance of the Haar measure, $U$ is symmetric, and it clearly 
contains the neutral element of $G$. 
Conversely, suppose that $U\subseteq E$ is symmetric, open and contains $e$. 
Since $G$ is second countable, it admits a left invariant metric \cite{k}, say $\rho$. 
Let $\delta > 0$ be such that $B(e,\delta)\subseteq U$. Letting $V = B(e,\delta/2)$, 
we have that $V\times V\subseteq U^*$ (indeed, if $s,t\in V$ then 
$$\rho(e,ts^{-1}) \leq \rho(e,t) + \rho(t,ts^{-1}) = \rho(e,t) + \rho(e,s^{-1}) < \delta).$$
It follows that $(Vr)\times (Vr)\subseteq U^*$ for all $r\in G$ and, letting $R\subseteq G$ be a 
countable dense set, we have that $\Delta\subseteq \cup_{r\in R} (Vr)\times (Vr)$. 
Thus, $\Delta\subseteq_{\omega} \sqinter(E^*)$, and
Theorem \ref{th_charomi} implies that $E^*$ is generated by squares.
\end{proof}

We are now in a position to show that the operator systems of the form 
$\cl S(\kappa)$, where $\kappa$ is a positivity domain, 
do not always contain positive rank one operators (see the paragraph after Definition \ref{d_pd}).

\begin{corollary}\label{c_exi}
There exists a positivity domain $\kappa$ such that 
$\cl S(\kappa)$ does not contain a positive rank one operator.
\end{corollary}
\begin{proof}
  Realise the group of the circle $\bb{T}$ as the interval $(-1,1]$,
  equipped with addition modulo $2$ and (normalised) Lebesgue measure;
  set $H = L^2(-1,$ $1)$.  Let $(t_n)_{n=1}^{\infty}\subseteq (0,1)$
  be a strictly decreasing sequence converging to zero and set
  \[E = \{0\}\cup
  \bigcup_{n=1}^{\infty} ([t_{2n},t_{2n-1}] \cup [-t_{2n-1},-t_{2n}]).\]
It is clear that $E$ does not contain a (symmetric) neighbourhood of $0$. 
By Theorem \ref{th_into}, $E^*$ is a positivity domain and it does not contain a non-null square.
By Proposition \ref{th_ro} (i), $\cl S(E^*)$ does not contain a positive rank one operator. 
\end{proof}

We will need the following result by Lau \cite{lau}. It will be used in the proof of 
Proposition \ref{p_proj}, and is the reason we require the assumption that $G$ be amenable 
in most of the subsequent results. 

\begin{lemma}\label{l_lau}
Let $G$ be a locally compact amenable group. Then there exists a net $(\xi_i)\subseteq L^2(G)$
of (positive) functions of norm one, such that the net $(P(\xi_i\otimes\overline{\xi_i}))$
is an approximate identity of $A(G)$.
\end{lemma}

If $E\subseteq G$ is a set that is the closure of its interior, we set
$$B(E) = \{u|_E : u\in B(G)\}.$$
If $E$ is in addition symmetric and contains the neutral element of $G$, 
we set 
$$B(E)^+ = \{u|_E : u\in B(G)^+\}.$$
We next prove an invariant version of  Proposition \ref{p_chsch} and 
Theorem \ref{th_ns}.

\begin{proposition}\label{p_proj} 
  Let $G$ be a locally compact amenable group satisfying the 
  Lebesgue Density Theorem, $E\subseteq G$ 
  be a set that is the closure of its interior and $u : E\to \bb{C}$ be a measurable function. 
The following hold:

(i) \ $N(u)$ is a Schur multiplier if and only if $u$ is equivalent to a function that lies in $B(E)$;

(ii) If $E$ is moreover symmetric and $e\in E$, then 
$N(u)$ is a positive Schur multiplier if and only if $u$ is equivalent to a function 
that lies in $B(E)^+$.
\end{proposition}
\begin{proof}
(i)  
Suppose that $u$ is equivalent to a function 
$u'$ in $B(E)$. We have that $u'$ is the restriction to $E$ of a function $v\in B(G)$; 
thus, $N(u')$ is the restriction to $E^*$ of the Schur multiplier $N(v)$ on $G\times G$.
By (\ref{eq_clinte}) in Theorem \ref{th_into}, $E^*$ is generated by rectangles, and by 
Proposition \ref{p_chsch}, $N(u')$ is a Schur multiplier. Since $N(u)\sim N(u')$, we 
conclude that $N(u)$ is a Schur multiplier. 

Conversely, suppose that $N(u)$ is a Schur multiplier on $E^*$. By Proposition \ref{p_chsch}, 
there is a Schur multiplier $\nph : G\times G\to \bb{C}$ whose restriction to $E^*$ is equivalent to $N(u)$. 
Let $(\xi_i)\subseteq L^2(G)$ be the net from Lemma~\ref{l_lau}, and 
$\psi_i = \xi_i\otimes\overline{\xi_i}$; thus, $\psi_i\in \cl T(G)^+$. 
We have that $\nph\psi_i\in \cl T(G)$ for all $i$. Set $v_i = P(\nph\psi_i)$; thus,
$(v_i)_i\subseteq A(G)$. Note that 
$$\|v_i\|_{A(G)} = \|P(\nph\psi_i)\|_{A(G)} \leq \|\nph\psi_i\|_{\cl T(G)} 
\leq \|\nph\|_{\frak{S}(G)}\|\psi_i\|_{\cl T(G)} = \|\nph\|_{\frak{S}(G)},$$
for every $i$. 
Without loss of generality, we may
assume that the net $(v_i)$ converges to a function $v\in B(G)$ 
in the weak* topology of $B(G)$ (here we view $B(G)$ as the dual of 
the C*-algebra $C^*(G)$ of $G$, see \cite{eymard}).

By~\cite[Lemma 2.3]{stt}, for almost every $t\in E$, we have 
$$v_i(t) = \int_G \nph(t^{-1}s,s)\psi_i(t^{-1}s,s) ds = u(t) \int_G \psi_i(t^{-1}s,s) ds \to u(t),$$
since $(P(\xi_i\otimes\overline{\xi_i}))$ converges to the constant function $1$ uniformly on compact sets.

Let $U\subseteq E$ be open. 
Denoting by $\lambda$ the left regular representation of $L^1(G)$ on $L^2(G)$, 
for every $f\in C_c(G)$ with support contained in $U$, 
we have 
$$\int_G v_i f dm = \langle \lambda(f),v_i\rangle \longrightarrow 
\langle \lambda(f),v\rangle = \int_G v f dm.$$
On the other hand, by the Lebesgue Dominated Convergence Theorem, 
$$\int_G v_i f dm \longrightarrow \int_G u f dm $$
and hence $\int_G v f dm = \int_G u f dm$. Since this holds for every $f\in C_c(G)$ supported on $U$, 
we conclude that $v = u$ almost everywhere on $U$, and since $U$ was an arbitrary open subset of $E$, 
we have that $v = u$ almost everywhere on $E$. It follows that 
$u$ is equivalent to a function from $B(E)$. 

(ii) By Theorem \ref{th_into} (i), $E^*$ is a positivity domain.
If $u\in B(E)^+$ then $u$ is the restriction to $E$ of a function $v\in B(G)^+$.
Thus, $N(v)$ is a positive Schur multiplier on $G\times G$
and $N(u)$ is its restriction to $E^*$. It follows that
$N(u)$ is a positive Schur multiplier.

Conversely, suppose that $N(u)$ is a positive Schur multiplier. 
By Theorem \ref{th_ns}, there exists 
$\nph\in \frak{S}(G)^+$ whose restriction to $E^*$ is equivalent to $N(u)$. 
But then (letting as above $v_i = P(\nph\psi_i)$), we have that $v_i\in A(G)^+$ for all $i$ 
and hence the weak* cluster point $v$ of the net $(v_i)$ is a positive definite function.
As in (i), we conclude that $u$ is equivalent to the restriction of $v$ to $E$.
\end{proof}

The next result gives a sufficient condition for the existence of 
positive definite extensions in operator-theoretic terms and is one of the 
main results of this section.

\begin{theorem}\label{th_lcex}
Let $G$ be a locally compact 
amenable group satisfying the Lebesgue Density Theorem. 
Let $E\subseteq G$ be a symmetric set which is the closure of its interior and 
contains the neutral element of $G$. 
Suppose that every positive operator in $\cl S(E^*)$
is a weak* limit of sums of rank one positive operators in $\cl S(E^*)$. 
If $u \in B(G)$ is positive definite on $E$ 
then there exists a positive definite function $v\in B(G)$ such that $v|_E = u|_E$. 
\end{theorem}
\begin{proof} 
Let $u_0 : E\to \bb{C}$ be the restriction of $u$ and $\nph = N(u_0)$. 
By Proposition \ref{p_proj}, 
$\nph$ is a Schur multiplier.
We claim that $\nph$ is partially positive. 
Indeed, suppose first that $\alpha\subseteq G$ is a compact subset such that 
$\alpha\times\alpha\subseteq E^*$. Since $\nph|_{\alpha\times\alpha} = N(u_0)|_{\alpha\times\alpha}$,
we have that $\nph|_{\alpha\times\alpha}$ 
is a continuous positive definite function. It follows that $\nph|_{\alpha\times\alpha}$
is a positive Schur multiplier 
(see the discussion preceding~\cite[Theorem 4.8]{naomi_tt} and the proof of that theorem).
If $\alpha\subseteq G$ is a measurable set of finite measure with 
$\alpha\times\alpha\subseteq E^*$ then, by the regularity of the Haar measure, 
there exists an increasing sequence $(\alpha_k)_{k\in \bb{N}}$ of compact subsets 
of $\alpha$ such that $\cup_{k\in \bb{N}}\alpha_k$ has full measure in $\alpha$. 
We have that $S_{\nph}(T) = \sotlim_{k\to\infty} S_{\nph}(P(\alpha_k)TP(\alpha_k))$ for 
every $T\in \cl B(P(\alpha)H)$; it thus follows that $\nph|_{\alpha\times\alpha}$
is a positive Schur multiplier. Finally, if $\alpha\subseteq G$ is an arbitrary measurable 
set such that $\alpha\times\alpha \subseteq E^*$ then by the $\sigma$-finiteness of 
the Haar measure, $\alpha$ is the union of an increasing sequence of sets of finite measure 
and it follows as before that $\nph|_{\alpha\times\alpha}$ is a positive Schur multiplier. 

We have thus shown that $\nph$ is a partially positive Schur multiplier. 
By Theorem~\ref{th_into}, $E^*$ is a positivity domain and, 
by Theorem~\ref{th_posr}, there exists a positive Schur multiplier $\psi : G\times G\to \bb{C}$
extending $\nph$. 
Now Proposition~\ref{p_proj} implies the existence of a positive definite function $v\in B(G)$ such that 
$v|_{E} = u_0$. 
The proof is complete.
\end{proof}

We note that the function $u$ in the statement of Theorem \ref{th_lcex} does not need to be defined 
on the whole of $G$ but just on the subset $E$, and it suffices that it possess an extension to a function from $B(G)$.

\medskip

We now turn our attention to discrete groups; the assumption on the second countability 
will be dropped for the rest of the section.

\begin{definition}
  Let $G$ be a discrete group.  A subset $E\subseteq G$ will
  be called a \emph{chordal subset of $G$} if $E$ is a symmetric set containing the neutral element $e$ with the following property: 
  if $n\geq 4$ and $s_1,\dots,s_n\in E$ with $s_n\cdots s_2 s_1 = e$, then there
  exist $i, k\in \{1, \ldots, n\}$ with $2\leq k-i \leq n-2$ so that $s_{k-1}s_{k-2}\cdots s_{i} \in E$.
\end{definition}

\begin{lemma}\label{l_grch}
Let $G$ be a discrete group and $E\subseteq G$ be a symmetric subset containing 
the neutral element. The following are equivalent:
\begin{enumerate}[(i)]
\item $E$ is a chordal subset of $G$;

\item the set $E^*$ is chordal.
\end{enumerate}
\end{lemma}

\begin{proof}
(i)$\Rightarrow$(ii) 
Suppose that $n\geq 4$ and consider a cycle
$$(x_1,x_2), \cdots, (x_{n-1},x_n), (x_n,x_1)\in E^*.$$
Setting 
$$s_1 = x_2 x_1^{-1}, \cdots, s_{n-1} = x_n x_{n-1}^{-1}  \mbox{ and } s_n = x_1 x_n^{-1},$$
we see that  $s_1, \ldots, s_n\in E$ and $s_n\cdots s_2 s_1 = e$. By the chordality of $E$, 
there exists $i, k \in \{1, \ldots, n\}$ with $2\leq k-i \leq n-2$ so that $x_k x_i^{-1}= s_{k-1} \cdots s_{i+1}s_i \in E$,
so $(x_i, x_k)\in E^*$ is a chord in the given cycle. 

(ii)$\Rightarrow$(i) Suppose that $n\geq 4$ and $s_1,\dots,s_n\in E$ with $s_n\cdots s_2 s_1 = e$.
The edges
$$(e,s_1), (s_1, s_2 s_1), (s_2 s_1, s_3 s_2 s_1), \cdots, (s_{n-1}\cdots s_1, e)$$
form a cycle of $E^*$, which we label as $(x_1, x_2), \cdots (x_{n-1}, x_n), (x_n, x_1)$. 
Since $E^*$ is chordal, there exist $1\leq i, k\leq n$ with $2\leq k-i \leq n-2$ so that 
$(x_i, x_k)\in E^*$, so $x_k x_i^{-1}= s_{k-1}\cdots s_{i+1}s_{i}\in E$. Hence $E$ is a chordal subset of $G$.
\end{proof}

As an application of our results on positive extensions of Schur multipliers, 
we now provide a different approach to the main result of \cite{bt}.

\begin{theorem}\label{th_chs}
If $E$ is a chordal subset of a discrete amenable group $G$, then
every positive definite function $u : E\to\bb{C}$ has a positive definite extension 
to a function $v : G\to \bb{C}$. 
\end{theorem}
\begin{proof}
Let $u : E\to\bb{C}$ be positive definite and $\nph(s,t) = N(u)(s,t) = u(ts^{-1})$, $(s,t)\in E^*$.
We note that $\nph$ is a Schur multiplier. 
Indeed, for every finite set $F\subseteq G$, let $\nph_F$ be the restriction of $\nph$
to $E^* \cap (F\times F)$. Then the map $S_{\nph_F}$ acts on $\cl S(E^*\cap (F\times F))$. 
Since $E$ is chordal, $\nph_F$ has a positive extension to a Schur multiplier $\psi_F$ on 
$F\times F$, and hence 
$$\|S_{\nph_F}\| \leq \|S_{\psi_F}\| = \|S_{\psi_F}(I)\| = \max_{s\in F}|\nph_F(s,s)| = u(e).$$ 
It follows that the Schur product by $\nph$ is a well-defined map from $\cl S(E^*)$ 
into $\cl B(H)$, and hence $\nph$ is a Schur multiplier. 

Note that, if $\alpha\subseteq G$ then $\alpha\times\alpha \subseteq E^*$ 
if and only if $\alpha \alpha^{-1}\subseteq E$. The positive definiteness of $u$ now implies
that $\nph$ is a partially positive Schur multiplier. 

By Lemma \ref{l_grch} and Theorem~\ref{thm:op-valued}, $\nph$ has an extension to a positive Schur multiplier
$\psi : G\times G\to \bb{C}$. By Proposition~\ref{p_proj}, there exists a 
positive definite extension $v$ of $u$. 
\end{proof}

As an illustration of Theorem \ref{th_chs}, note that every 
symmetric arithmetic progression $E\subseteq \bb{Z}$ containing $0$
is a chordal set; thus, Theorem \ref{th_chs} applies and gives the well-known 
fact that every positive definite function $u : E\to\bb{C}$ has a positive definite extension 
to a function $v : \bb{Z}\to \bb{C}$ (see \cite[Theorem 4.8]{g}).

\bigskip

\noindent {\bf Acknowledgements} 
We wish to thank Yemon Choi for bringing the reference \cite{bt}
to our attention. 
We also thank the anonymous referee for the careful reading of the manuscript and his/her 
useful suggestions.

\end{document}